\documentclass[11pt, a4paper]{article}

\usepackage{enumerate}
\usepackage{url}

\usepackage{color}
\usepackage{epsfig}
\usepackage{graphicx} % Required for inserting images
\usepackage{caption}
\usepackage{amsfonts}
\usepackage{amssymb}
\usepackage{mathrsfs}
\usepackage{amsmath}
\usepackage{amsthm}
\usepackage{enumerate}
\usepackage{graphics}
\usepackage{pict2e}
\usepackage{cite}
\usepackage{subfigure}
\usepackage[subnum]{cases}
\usepackage{multirow}
\usepackage{comment}
\usepackage{xcolor}
\usepackage{cleveref}
\usepackage{authblk}
\usepackage{cleveref}
\usepackage{diagbox}
\usepackage{slashbox}

\def\Proof{\par\noindent{\em Proof of Theorem}}

\usepackage{xcolor}
\usepackage[normalem]{ulem}

\makeatletter

\newtheorem{theorem}{Theorem}

\newtheorem{lemma}[theorem]{Lemma}

\theoremstyle{definition}
\newtheorem{definition}{Definition}

\newtheorem{conjecture}[theorem]{Conjecture}

\newcommand{\A}{\mathcal{A}}
\newcommand{\I}{\mathcal{I}}

\title{Odd Hadwiger's conjecture for the complements of Kneser graphs}

\author[1]{Meirun Chen}
\author[2]{Reza Naserasr}
\author[3]{Lujia Wang}
\author[4]{Sanming Zhou}

\affil[1]{\small Xiamen University of Technology, Xiamen, China.  {Email: \texttt{mrchen@xmut.edu.cn}}}
\affil[2]{\small Université Paris Cité, CNRS, IRIF, F-75013, Paris, France. {Email: \texttt{reza@irif.fr}}}
\affil[3]{\small Zhejiang Normal University, Jinhua, China. {Email: \texttt{ljwang@zjnu.edu.cn}}}
\affil[4]{School of Mathematics and Statistics, The University of Melbourne, Parkville, VIC 3010,  Australia. {Email: \texttt{sanming@unimelb.edu.au}}}

\begin{document}
	
	\maketitle

	\begin{abstract}
A generalization of the four-color theorem, Hadwiger's conjecture is considered as one of the most important and challenging problems in graph theory, and odd Hadwiger's conjecture is a strengthening of Hadwiger's conjecture by way of signed graphs. In this paper, we prove that odd Hadwiger's conjecture is true for the complements $\overline{K}(n,k)$ of the Kneser graphs $K(n,k)$, where $n\geq 2k \ge 4$. This improves a result of G. Xu and S. Zhou (2017) which states that Hadwiger's conjecture is true for this family of graphs. Moreover, we prove that $\overline{K}(n,k)$ contains a 1-shallow complete minor of a special type with order no less than the chromatic number $\chi(\overline{K}(n,k))$, and in the case when $7 \le 2k+1 \le n \le 3k-1$ the gap between the odd Hadwiger number and chromatic number of $\overline{K}(n,k)$ is $\Omega(1.5^{k})$. 
	\end{abstract}

\section{Introduction} 
\label{sec:int}
	
A graph $H$ is said to be a \emph{minor} of a graph $G$ if it can be obtained from $G$ by a series of operations composed of (i) deleting a vertex, (ii) deleting an edge, and (iii) contracting an edge. In such a case $G$ is said to contain an \emph{$H$-minor}. Otherwise $G$ is said to be \emph{$H$-minor-free}. An $H$-minor of $G$ can be thought of as a family of pairwise vertex-disjoint connected subgraphs of $G$, each called a \emph{bag}, such that the graph obtained from $G$ by deleting all vertices outside these subgraphs and contracting all edges within each subgraph is isomorphic to $H$. The characterization of planar graphs as $K_5$- and $K_{3,3}$-minor-free graphs has inspired the study of relations between colorings and minors of graphs. In this regard Wagner~\cite{kW37} proved that if the four color theorem holds, then every $K_5$-minor-free graph is $4$-colorable. This motivated Hadwiger to conjecture that every $K_{t}$-minor-free graph is $(t-1)$-colorable. The \emph{Hadwiger number} of a graph $G$, denoted by $h(G)$, is the largest $t$ such that $K_t$ is a minor of $G$. As usual, we use $\chi(G)$ to denote the chromatic number of $G$. Hadwiger's conjecture can be stated as follows.
	
	\begin{conjecture}
		[Hadwiger's conjecture] 
		\label{Conj:Hadwiger}
		Every graph $G$ satisfies $\chi(G) \leq h(G)$.
	\end{conjecture}
	
	This conjecture is considered as one of the most central problems in graph theory. It has been proved for graphs with $\chi(G) \leq 6$ (see~\cite{RST93} for the proof when $\chi(G) = 6$). In the general case, a common approach towards establishing Hadwiger's conjecture is to prove good upper bounds on the chromatic number of $K_t$-minor-free graphs in terms of $t$; see~\cite{kK09} and the references therein for results obtained by using this approach. Since it is challenging to prove Hadwiger's conjecture in its general form, it is natural to attempt it for special classes of graphs; see, for example, \cite{bc09, CIZ19, ll07, RST93, wxz, xz16} for several results along this direction. 
	
	Another venue of research is to consider possible strengthenings of Hadwiger's conjecture by way of signed graphs. The importance of this direction lies in the following algorithmic consideration. Given a positive integer $t$, it can be decided in polynomial time if an input graph $G$ has a $K_t$-minor or not. Upon receiving a ``NO'' answer for an input graph $G$, if Hadwiger's conjecture is verified for the class of $K_t$-minor-free graphs, then the existence of a $(t-1)$-coloring of $G$ is assured, whereas the decision problem of deciding whether a general graph is $(t-1)$-colorable is NP-complete for any fixed $t\geq 4$. As $K_t$-minor-free graphs have bounded average degree, this approach, unfortunately, applies only to families of sparse graphs. For example, $K_3$-minor-free graphs are forests, though in this particular case verifying 2-colorability of any graph can be done in polynomial time. To strengthen Hadwiger's conjecture so that it works for a larger class of graphs, the notion of signed graphs and their minors are employed.   
	
A \emph{signed graph} $(G,\sigma)$ is a graph $G$ together with a function $\sigma$ that assigns a sign $\sigma(uv)$ ($+$ or $-$) to each edge $uv$ of $G$.  A \emph{switching} at a vertex $x$ of $(G,\sigma)$ is to multiply the sign of all edges incident with $x$ by a negative sign (that is, change the sign of each edge incident with $x$). Any signed graph obtained from $(G,\sigma)$ by a series of switchings is said to be \emph{switching equivalent} to $(G,\sigma)$. A cycle or closed walk $C$ of $(G,\sigma)$ is said to be \emph{negative} or \emph{unbalanced} if it contains an odd number of negative edges; otherwise it is \emph{positive} or \emph{balanced}. In other words, $C$ is negative or positive depending on whether its sign is negative or positive, where the \emph{sign} of $C$ is the product of the signs of its edges. The concept of graph minors has been extended to signed graphs as follows. A signed graph $(H,\pi)$ is said to be a \emph{minor} of a signed graph $(G,\sigma)$ if $(H,\pi)$ can be obtained from $(G,\sigma)$ by a series of the following operations: (i) deleting a vertex, (ii) deleting an edge, (iii) contracting a positive edge, (iv) switching at a vertex. Equivalently, $(H,\pi)$ is said to be a minor of $(G,\sigma)$ if there exist a signed graph $(G, \tau)$ switching-equivalent to $(G, \sigma)$ and $|V(H)|$ pairwise vertex-disjoint subgraphs $B_x$ of $G$, where $x \in V(H)$, such that (i) for each $x \in V(H)$, the set of positive edges of $B_x$ under $\tau$ induces a connected spanning subgraph of $B_x$, and (ii) for each edge $xy$ of $H$, there exists an edge $uv$ of $G$ with $u \in V(B_x)$ and $v \in V(B_y)$ such that $\tau(uv) = \pi(xy)$. We call $B_x$ the \emph{bag} corresponding to $x$. This equivalent definition is based on the fact that one can obtain $(H,\pi)$ from $(G, \tau)$ by deleting all vertices outside these bags, deleting all edges between $B_x$ and $B_y$ except for one edge $uv$ with $\tau(uv) = \pi(xy)$ for each $xy$ of $H$, and contracting all positive edges under $\tau$ in each bag.  
 
Denote by $(G, -)$ be the signed graph with underlying graph $G$ in which all edges are negative. Regarding each graph $G$ as the signed graph $(G,-)$, the class of graphs becomes a subclass of the class of signed graphs. This point of view has led Gerards and Seymour to introduce independently a strengthening of Hadwiger's conjecture which is now known as odd Hadwiger's conjecture (see \cite{JT95, S16}). It was motivated by an earlier result of Catlin in~\cite{pC79} which can be considered as the first nontrivial case of the conjecture.
	
	\begin{conjecture}
		[Odd Hadwiger's conjecture] 
		\label{conj:OddHadwiger} If $(G,-)$ is $(K_t, -)$-minor-free, then $\chi(G)\leq t-1$.
	\end{conjecture}
	
	Note that in most literature a graph $G$ for which $(G,-)$ has no $(K_t, -)$-minor is said to be odd $K_t$-minor-free and this definition is given without using signatures. In such a definition, vertices are 2-colored (corresponding to whether a switching is done at the vertex or not), contracted edges are bicolored (hence positive), and edges representing the connection of bags are monochromatic (hence negative). The above formulation of odd Hadwiger's conjecture based on the notion of signed graphs is taken from~\cite{NRS15}. To better present our work, we will restate odd Hadwiger's conjecture using the following definition.
	
	\begin{definition}
		The \emph{odd Hadwiger number} of a graph $G$, denoted $h_o(G)$, is the largest integer $t$ such that $(K_t, -)$ is a minor of $(G, -)$.  
	\end{definition}
	
	\begin{conjecture}
		[Odd Hadwiger's conjecture restated]
		Every graph $G$ satisfies $\chi(G)\leq h_o(G)$.
	\end{conjecture}
	
	Observe that if $(K_t,-)$ is a minor of a signed graph $(G, \sigma)$, then $K_t$ is a minor of $G$. In other words, if $G$ is $K_t$-minor-free, then $(G,-)$ is $(K_t, -)$-minor-free. (Note that the converse is not necessarily true.) So we have $h_o(G) \le h(G)$. Thus, if $G$ satisfies odd Hadwiger's conjecture, then it satisfies Hadwiger's conjecture. Furthermore, if true, odd Hadwiger's conjecture bounds the chromatic number of a much larger class of graphs which, in particular, contains the class of bipartite graphs, whereas Hadwiger's conjecture provides no bound for the chromatic number of this basic class. Hence in the study of odd Hadwiger's conjecture one bounds the chromatic number of a family of graphs some of whose members are dense graphs, whereas in the study of Hadwiger's conjecture one can only bound the chromatic number of sparse graphs. It is not hard to verify that in the case $t = 3$ odd Hadwiger's conjecture says that every graph with no odd cycle is 2-colorable, which is trivially true. The case $t=4$ is the above-mentioned result of Catlin \cite{pC79}. The case $t=5$ implies the four color theorem. A proof in this case, using the four color theorem, was announced by Gueinin in 2004 but no publication is available yet. All remaining cases are open. 
	
	While odd Hadwiger's conjecture is stronger than Hadwiger's conjecture, an asymptotic inverse implication is provided recently in \cite{S22}. More specifically, it is shown by Steiner \cite{S22} that if every $K_t$-minor-free graph admits an $f(t)$-coloring for some function $f$, then graphs with no odd $K_t$-minors admit a $2f(t)$-coloring. In particular, if Hadwiger's conjecture is proved, then it would imply that every odd $K_t$-minor-free graph is $2(t-1)$-colorable. See \cite{JMNNQ24+} for more connections between odd Hadwiger's conjecture and colorings of signed graphs.
	
A strengthening of Hadwiger's conjecture, Haj\'{o}s' conjecture says that every graph $G$ with $\chi(G) \ge t$ contains a subdivision of $K_t$. This conjecture was proved to be false for every $t \ge 7$ by Catlin \cite{pC79}. In \cite{thom}, Thomassen gave several new classes of counterexamples to Haj\'os' conjecture, including the complements of the Kneser graphs $K(3k-1, k)$ for sufficiently large $k$. (Given integers $n$ and $k$ with $n \ge 2k \ge 4$, the \emph{Kneser graph} $K(n,k)$ is the graph whose vertices are the $k$-subsets of $[n] = \{1, 2, \ldots, n\}$ in which two such vertices are adjacent if and only if they are disjoint.) Meanwhile, he noticed that ``it does not seem obvious" that these classes all satisfy Hadwiger's conjecture. As a follow-up to this remark, Xu and Zhou \cite{xz16} proved that indeed Hadwiger's conjecture is true for the complement of every Kneser graph. Since odd Hadwiger's conjecture is stronger than Hadwiger's conjecture, it would be natural to ask whether the complements of all Kneser graphs satisfy odd Hadwiger's conjecture. In this paper we prove that this is indeed the case. In a special case, our result confirms odd Hadwiger's conjecture for a particular family of graphs with independence number 2, namely the complements of Kneser graphs $K(n,k)$ with $2k+1 \le n \le 3k-1$. 
%In general, as with the case of Hadwiger's conjecture (see \cite{co10} for a related discussion), it is reasonable to expect that odd Hadwiger's conjecture for graphs of independence number 2 is challenging as such graphs have large chromatic numbers. 
Among them the complement of $K(3k-1, k)$ is especially interesting because not only does it provide a counterexample to Haj\'{o}s' conjecture for sufficiently large $k$ but also the gap between its chromatic number $\binom{3k-1}{k-1}$ and clique number $\binom{3k-2}{k-1}$ can be arbitrarily large when $k$ approaches infinity. In general, as with the case of Hadwiger's conjecture, it is expected that the larger the gap between the chromatic number and clique number is, the harder it is to verify odd Hadwiger's conjecture.

A \emph{star} is the complete bipartite graph $K_{1, s}$ for some $s \ge 1$. A minor of a graph is said to be $r$-shallow if each of its bags has radius at most $r$. A minor is said to be \emph{strongly $1$-shallow} if each of its bags is either a single vertex or a star, the edge representing the connection between two star bags joins leaf vertices of the bags, and the edge representing the connection between a star bag and a single-vertex bag joins a leaf vertex of the star bag and the unique vertex of the single-vertex bag. Denote by $\overline{K}(n,k)$ the complement of $K(n,k)$. Our main result, stated below, shows that not only does $\overline{K}(n,k)$ satisfy odd Hadwiger's conjecture but also it admits a strongly $1$-shallow $K_t$-minor for some $t$ no less than $\chi(\overline{K}(n,k))$. Moreover, for $2k+1 \le n \le 3k-1$, the gap between the odd Hadwiger number and chromatic number of $\overline{K}(n,k)$ is $\Omega(1.5^{k})$; in particular, this gap can be arbitrarily large as $k \rightarrow \infty$. 
	
\begin{theorem}\label{thm:OddHadwigerForCompKneser}
For any integers $n$ and $k$ with $n\geq 2k \ge 4$, we have 
$$
\chi (\overline{K}(n,k)) \leq h_o(\overline{K}(n,k)).
$$ 
Moreover, there exists an integer $t = t(n, k) \geq \chi(\overline{K}(n,k))$ such that $\overline{K}(n,k)$ contains a strongly $1$-shallow $K_t$-minor. Furthermore, for $k \ge 3$, we have
$$
h_o(\overline{K}(n,k)) - \chi (\overline{K}(n,k)) \ge  
\begin{cases}
\frac{1}{12}(1.5^{k-1}-8k-2), & \text{ if }\ 2k+1 \le n \le 3k-2 \\
\frac{1}{6}(1.5^{k-1}-11), & \text{ if }\ n = 3k-1.
\end{cases}
$$
\end{theorem}

Since $\overline{K}(n,2)$ is isomorphic to the line graph of $K_n$, the statement that $\chi (\overline{K}(n,2)) \leq h_o(\overline{K}(n,2))$ is covered by a recent result of Steiner [18] which establishes odd Hadwiger's conjecture for the line graphs of simple graphs.

The strongly $1$-shallow $K_t$-minor of $\overline{K}(n,k)$ guaranteed by Theorem \ref{thm:OddHadwigerForCompKneser} gives rise to a $(K_t, -)$-minor of $(\overline{K}(n,k), -)$ with the same bags and connections (see Lemma \ref{lem:strong1shallow} and its proof). It is worth mentioning that Bollobás, Catlin and Erd\H{o}s \cite{BCE80} proved that Hadwiger's conjecture holds for almost all graphs. In fact, they proved the existence of a strongly 1-shallow $K_t$-minor in almost every $t$-chromatic graph, so they actually established odd Hadwiger's conjecture for almost all graphs. However, even if we know a $t$-chromatic graph satisfies odd Hadwiger's conjecture, in general it is difficult to find a strongly 1-shallow $K_t$-minor in it or prove such a $K_t$-minor does not exist. In our proof of Theorem \ref{thm:OddHadwigerForCompKneser}, we will give an explicit construction of such a strongly 1-shallow $K_t$-minor in $\overline{K}(n,k)$.
	
The rest of this paper is organized as follows. In the next section we introduce notation and present five preliminary results, including two lemmas which prove Theorem \ref{thm:OddHadwigerForCompKneser} when $k = 2$ and $k$ is a divisor of $n$, respectively. In Section \ref{sec:const}, we give three constructions of strongly 1-shallow $K_t$-minors of $\overline{K}(n,k)$ using a lemma from \cite{xz16} (see Lemma \ref{lem:PartitionA_i} in the next section). In Section \ref{sec:pf}, we prove that for each pair of integers $n, k$ with $n \ge 2k + 1 \ge 7$ at least one of these constructions gives a strongly 1-shallow $K_t$-minor with $t \ge \chi(\overline{K}(n,k))$, thereby completing the proof of Theorem \ref{thm:OddHadwigerForCompKneser}.

\section{Preliminaries}

\begin{lemma}
\label{lem:strong1shallow}
If a graph $G$ contains a strongly $1$-shallow $K_t$-minor, then $h_o(G)\geq t$.
\end{lemma}

\begin{proof}
Suppose that $H$ is a strongly 1-shallow $K_t$-minor of $G$. Then switching at the center vertex of each star bag of $(H,-)$ results in a $(K_t,-)$-minor of $(G,-)$. Hence $h_o(G)\geq t$.
\end{proof}
	
In what follows we always assume $n$ and $k$ are integers with $n \ge 2k \ge 4$. Set $[n] = \{1, 2, \ldots, n\}$ and denote by $\binom{[n]}{k}$ the set of all $k$-subsets of $[n]$. Then both $K(n,k)$ and $\overline{K}(n,k)$ have vertex set $\binom{[n]}{k}$.  

\begin{lemma} 
\label{lem:divisor}
For any integers $s, k \ge 2$, we have $\chi (\overline{K}(sk,k)) \leq h_o(\overline{K}(sk,k))$, and $\overline{K}(sk,k)$ contains a strongly $1$-shallow $K_t$-minor, where $t = \chi(\overline{K}(sk,k))$. 
\end{lemma}  

\begin{proof}	
The chromatic number and clique number of $\overline{K}(sk,k)$ are equal. Turning each vertex in a maximum clique of $\overline{K}(sk,k)$ into a single-vertex bag, we obtain a strongly 1-shallow $K_{t}$-minor of $\overline{K}(sk,k)$ where $t = \chi(\overline{K}(sk,k))$. Hence the result follows. 
\end{proof}

\begin{lemma} 
\label{lem:2}
For any integer $n \ge 4$, we have $\chi (\overline{K}(n,2)) \leq h_o(\overline{K}(n,2))$, and $\overline{K}(n,2)$ contains a strongly $1$-shallow $K_t$-minor, where $t = n-1 = \chi (\overline{K}(n,2))$ when $n$ is even and $t = n = \chi (\overline{K}(n,2))$ when $n$ is odd.
\end{lemma}  

\begin{proof}	
If $n$ is even, then $\chi(\overline{K}(n,2)) = n-1$ and the result follows from Lemma \ref{lem:divisor}. If $n$ is odd, then $\chi (\overline{K}(n,2)) = n$, and we can obtain a strongly 1-shallow $K_n$-minor of $\overline{K}(n,2)$ by treating each vertex $\{1, i\}$, $2 \le i \le n$ as a single-vertex bag and taking the star bag with center vertex $\{2, 4\}$ and leaf vertices $\{2,3\}, \{3, 4\}, \{2, 5\}, \{2, 6\}, \ldots, \{2, n\}$. So the result is also true when $n$ is odd.  
\end{proof}

The following formula for the chromatic number of $\overline{K}(n,k)$ was obtained by Baranyai \cite{bar73} in 1973. We will use it in our proof of Theorem \ref{thm:OddHadwigerForCompKneser}.
	
	\begin{lemma}[\cite{bar73}]
		\label{lem:Baranyai} 
		$\chi (\overline{K}(n,k)) = \left\lceil \binom{n}{k} / {\lfloor \frac{n}{k}\rfloor} \right\rceil$.
	\end{lemma}    
	
For $1 \le i \le n-2k+1$, let 
$$
\A_i(n,k) = \left\{A \in \binom{[n]}{k}: i \in A,\, A \setminus \{i\} \subseteq \{i+1, i+2, \ldots, n\}\right\}.
$$ 
Clearly, $|\A_i(n,k)| = \binom{n-i}{k-1}$ and each member of $\A_i(n,k)$ is a vertex of $\overline{K}(n,k)$. Since all members of $\A_i(n,k)$ contain $i$, $\A_i(n,k)$ is a clique of $\overline{K}(n,k)$. As a side remark, by the well-known Erd\H{o}s-Ko-Rado theorem, $\A_i(n,k)$, $i=1,2, \ldots, n-2k+1$, are maximum cliques of $\overline{K}(n,k), \overline{K}(n-1,k), \ldots, \overline{K}(2k,k)$, where the ground set of $\overline{K}(n-i+1,k)$ is $\{i, i+1, \ldots, n\}$. The following lemma is one of the key tools for our proof of Theorem \ref{thm:OddHadwigerForCompKneser}.
	
	\begin{lemma}[{\cite[Lemma 4]{xz16}}]\label{lem:PartitionA_i}
		Given positive integers $i$ and $l$ with $i \le n-k+1$, let $c_i = \lfloor {n-i \choose k-1}/l \rfloor$. Then $\A_i(n,k)$ can be partitioned into $\A_{i1}^l(n,k), \A_{i2}^l(n,k), \ldots, \A_{i{c_i}}^l(n,k)$ each of size $l$, together with $\A_{i,{c_i}+1}^l(n,k)$ of size ${n-i \choose k-1} - {c_i}l$ when ${n-i \choose k-1}$ is not divisible by $l$, such that for $1\leq j \leq c_i$,  
\begin{equation}
\label{eq:min}
|\cup_{A \in \A_{ij}^l(n,k)} A| \ge \min\{n-i+1, l(k-1)+1\}.
\end{equation}
	\end{lemma}    
	
The following well-known identity will be used to simplify formulas in our proof of Theorem~\ref{thm:OddHadwigerForCompKneser}: for integers $a \ge b \ge 0$,
$$ 
\sum_{i=0}^{a}{i \choose b}={a+1 \choose b+1}.
$$

\section{Building strongly $1$-shallow $K_t$-minors in $\overline{K}(n,k)$}
\label{sec:const}

In view of Lemmas  \ref{lem:divisor} and \ref{lem:2}, from now on we assume 
$$
n=sk+r, \text{ where $s\geq 2$, $k\geq 3$ and $1 \le r\leq k-1$}.
$$ 
	
	\subsection{Construction 1}%Old 3.1 r\leq k-s

Choose $l$ to be the smallest integer such that $l(k-1)+1\geq n-k+1$, that is, $l= \lceil\tfrac{n-k}{k-1}\rceil$, and set $c_i = \lfloor \frac{1}{l}{n-i \choose k-1} \rfloor$ for $1\le i \le n-k+1$. Since $k < n-k+1$, by Lemma~\ref{lem:PartitionA_i}, for $2 \le i \le k$, we can partition $\A_i(n,k)$ into $\A_{ij}^l(n,k)$, $1 \le j \le c_i$, each with size $l$, together with $\A_{i, c_i+1}^l(n,k)$ with size ${n-i \choose k-1} - {c_i}l < l$ when ${n-i \choose k-1}$ is not divisible by $l$, such that \eqref{eq:min} holds for $1 \le j \le c_i$. Since $l(k-1)+1\geq n-k+1$, by \eqref{eq:min}, we have $|\cup_{A \in \A_{ij}^l(n,k)} A| \ge \min\{n-i+1, l(k-1)+1\} \ge \min\{n-i+1, n-k+1\} = n-k+1$ for $2 \le i \le k$ and $1 \le j \le c_i$. This implies that each member of $\binom{[n]}{k}$ has a non-empty intersection with $\cup_{A \in \A_{ij}^l(n,k)} A$ and hence is adjacent to at least one member of $\A_{ij}^l(n,k)$ in $\overline{K}(n,k)$. In other words, $\A_{ij}^l(n,k)$ is a dominating set of $\overline{K}(n,k)$.

Set $d_i= \lfloor \frac{1}{l+1}{n-i \choose k-1}\rfloor$ for $2 \le i \le k$. Then $d_i \le c_i$ and $|\A_i(n,k) - \cup_{j=1}^{d_i} \A_{ij}^l(n,k)| = {n-i \choose k-1} - d_{i} l = {n-i \choose k-1} - d_{i} (l+1) + d_{i} \ge d_i$. So we can take $d_i$ distinct members of $\A_i(n,k) - \cup_{j=1}^{d_i} \A_{ij}^l(n,k)$, say, $A_1, A_2, \ldots, A_{d_i}$. For $1 \le j \le d_i$, construct the star bag in $\overline{K}(n,k)$ with $A_j$ as its centre and members of $\A_{ij}^l(n,k)$ as its leaf vertices. This star is in $\overline{K}(n,k)$ as $\A_i(n,k)$ is a clique of $\overline{K}(n,k)$. Thus, we get $d_i$ star bags in $\overline{K}(n,k)$ for each $2 \le i \le k$. Now we add each member of $\A_1(n,k)$ as a single-vertex bag. In this way we have constructed $t_1(n,k) = |\A_1(n,k)| + \sum_{i=2}^k d_i$ bags in total. Since $\A_1(n,k)$ is a clique of $\overline{K}(n,k)$ and each $\A_{ij}^l(n,k)$ is a dominating set of $\overline{K}(n,k)$, these bags give rise to a strongly 1-shallow complete minor of $\overline{K}(n,k)$ with order $t_1(n,k)$. We have
	\begin{eqnarray}
		t_1(n,k) & = & |\A_1(n,k)| + \sum_{i=2}^k d_i \nonumber \\ 
		& =   & {n-1 \choose k-1} + \sum_{i=2}^k \left\lfloor \frac{{n-i \choose k-1}}{l+1}\right\rfloor \nonumber \\
		& \ge & {n-1 \choose k-1} + \frac{1}{l+1} \sum_{i=2}^k {n-i \choose k-1} - 
		(k-1)\frac{l}{l+1} \nonumber \\ 
		& =  & {n-1 \choose k-1} + \frac{1}{l+1}\left({n-1 \choose k} -{n-k \choose 
			k}\right) -(k-1)\frac{l}{l+1}. 
		\label{eq:t<k-s}
	\end{eqnarray}

	\subsection{Construction 2} %Old 3.2 for k and s large

Choose $l=\lceil \frac{n-1}{2(k-1)}\rceil$ and set $c_i = \lfloor \frac{1}{l}{n-i \choose k-1} \rfloor$ for $1\le i \le n-k+1$. Since $\lceil \frac{n}{2} \rceil \le \frac{n+1}{2} < n-k+1$, by Lemma~\ref{lem:PartitionA_i}, for $1 \le i \le \lceil \frac{n}{2} \rceil$, we can partition $\A_i(n,k)$ into $\A_{ij}^l(n,k)$, $1 \le j \le c_i$, each with size $l$, together with $\A_{i, c_i+1}^l(n,k)$ with size ${n-i \choose k-1} - {c_i}l < l$ when ${n-i \choose k-1}$ is not divisible by $l$, such that \eqref{eq:min} holds for $1 \le j \le c_i$. By our choice of $l$, we have $l(k-1)+1 \geq \frac{n+1}{2}$. Thus, by \eqref{eq:min}, we have $|\cup_{A \in \A_{ij}^l(n,k)} A| \ge \min\{n-i+1, l(k-1)+1\} \ge \min\left\{n-i+1, \frac{n+1}{2}\right\} = \frac{n+1}{2}$ for $1 \le i \le \lceil \frac{n}{2} \rceil$ and $1 \le j \le c_i$.

Set $d_i= \lfloor \frac{1}{l+1}{n-i \choose k-1}\rfloor$ for $1 \le i \le \lceil \frac{n}{2} \rceil$. As with Construction 1, we have $d_i \le c_i$ and $|\A_i(n,k) - \cup_{j=1}^{d_i} \A_{ij}^l(n,k)|  \ge d_i$. So we can take $d_i$ distinct members $A_{i1}, A_{i2}, \ldots, A_{i d_i}$ of $\A_i(n,k) - \cup_{j=1}^{d_i} \A_{ij}^l(n,k)$. For $1 \le j \le d_i$, construct the star bag in $\overline{K}(n,k)$ with $A_{ij}$ as its centre and members of $\A_{ij}^l(n,k)$ as its leaf vertices. This star is in $\overline{K}(n,k)$ as $\A_i(n,k)$ is a clique of $\overline{K}(n,k)$. In this way we have constructed $t_2(n,k) = \sum_{i=1}^{\lceil n/2 \rceil} d_i$ star bags in total. Note that for any (not necessarily distinct) $i, p$ with $1 \le i, p \le \lceil \frac{n}{2} \rceil$ and any (not necessarily distinct) $j, q$ with $1 \leq j \leq d_i, 1 \leq q \leq d_{p}$ such that $(i, j) \ne (p, q)$, we have $|(\cup_{A \in \A_{ij}^l(n,k)} A) \cap (\cup_{A \in \A_{pq}^l(n,k)} A)| = |\cup_{A \in \A_{ij}^l(n,k)} A| + |\cup_{A \in \A_{pq}^l(n,k)} A| - |(\cup_{A \in \A_{ij}^l(n,k)} A) \cup (\cup_{A \in \A_{pq}^l(n,k)} A)| \ge \frac{n+1}{2} + \frac{n+1}{2} - n = 1$. Thus, $(\cup_{A \in \A_{ij}^l(n,k)} A) \cap (\cup_{A \in \A_{pq}^l(n,k)} A) \neq \emptyset$ and consequently there exists an edge of $\overline{K}(n,k)$ joining a leaf vertex of the star bag centred at $A_{ij}$ and a leaf vertex of the star bag centred at $A_{pq}$. Therefore, these $t_2(n,k)$ star bags produce a strongly 1-shallow complete minor of $\overline{K}(n,k)$ with order $t_2(n,k)$. We have
	\begin{align}
		t_2(n,k)= & \sum_{i=1}^{\lceil n/2 \rceil} \left\lfloor \frac{{n-i \choose k-1}}{l+1}\right\rfloor \nonumber \\
		\geq & \sum_{i=1}^{\lceil n/2\rceil} \frac{{n-i \choose k-1}}{l+1} - \frac{l}{l+1} \left\lceil\frac{n}{2}\right\rceil \nonumber \\
		= & \frac{1}{l+1}\left({n \choose k}-{n-\lceil\frac{n}{2}\rceil \choose k}\right)-\frac{l}{l+1}\left\lceil\frac{n}{2}\right\rceil \nonumber \\
		= & \frac{1}{l+1}\left({n \choose k}-{\lfloor\frac{n}{2}\rfloor \choose k}\right)-\frac{l}{l+1}\left\lceil\frac{n}{2}\right\rceil. \label{eq:t2}
	\end{align}
	
	\subsection{Construction 3}

We assume $s\leq k$ for this construction. We first treat each member of $\A_1(n,k)$ as a single-vertex bag and then construct a number of star bags in the subgraph $\overline{K}(n-1, k)$ (with ground set $\{2, 3, \ldots, n\}$) of $\overline{K}(n, k)$. Since $n = sk+r$ where $1 \le r \le k-1$, the independence number of $\overline{K}(n, k)$ is $s$. Since $r \geq 1$, we have $n-1 = sk+(r-1)$ with $0 \le r-1 < k-1$ and hence the independence number of $\overline{K}(n-1, k)$ is also $s$. Moreover, any maximum independent set $\I = \{A_1, A_2, \ldots, A_s\}$ of $\overline{K}(n-1, k)$ is a dominating set of $\overline{K}(n,k)$ as $\cup_{j=1}^s A_j$ being an $sk$-subset of $[n]$ has a non-empty intersection with every $k$-subset of $[n]$. To construct a star bag from $\I = \{A_1, A_2, \ldots, A_s\}$, we look for a vertex of $\overline{K}(n-1, k)$ which is adjacent to all members of $\I$ in $\overline{K}(n-1, k)$. Since such a vertex is a $k$-subset of $\{2, 3, \ldots, n\}$ which has a non-empty intersection with each of the $s$ pairwise disjoint sets $A_1, A_2, \ldots, A_s$, its existence requires $s \leq k$. A vertex of $\overline{K}(n-1, k)$ fails this condition if and only if it is disjoint from at least one of these $s$ sets; the number of such vertices depends on $n$ and $k$ only and is denoted by $\overline{d}(n,k)$. Note that there are ${n-k-1\choose k}$ $k$-subsets of $\{2, 3, \ldots, n\}$ which are disjoint from any given $A_i \in \I$. Since there are $s$ members in $\mathcal{I}$, it follows that $\overline{d}(n,k)\leq s{n-k-1 \choose k}$.
	
We now choose $m$ pairwise disjoint maximum independent sets $\I_1, \I_2, \ldots, \I_m$ of $\overline{K}(n-1, k)$ such that $\overline{K}(n-1, k) - \cup_{i=1}^m \I_i$ has at least $m+\overline{d}(n,k)$ vertices, that is, ${n-1 \choose k} - ms \ge m+\overline{d}(n,k)$. The largest integer with this property is $m = \lfloor\frac{1}{s+1}\left( {n-1 \choose k} - \overline{d}(n,k)\right)\rfloor$. By this choice of $m$ and the definition of $\overline{d}(n,k)$, there exist $m$ distinct vertices $B_1, B_2, \ldots, B_m$ of $\overline{K}(n-1, k)$ such that $B_i$ is adjacent to all members of $\I_i$ in $\overline{K}(n-1, k)$, for $1 \le i \le m$. So we can construct a star bag with centre $B_i$ and leaf vertices the members of $\I_i$, for $1 \le i \le m$. Along with the ${n-1\choose k-1}$ single-vertex bags from $\A_1(n,k)$, we have constructed $t_3(n,k) = {n-1\choose k-1} + m$ bags. Since $\A_1(n,k)$ is a clique of $\overline{K}(n, k)$ and each $\I_i$ is a dominating set of $\overline{K}(n,k)$, these bags produce a strongly 1-shallow complete minor of $\overline{K}(n, k)$ with order $t_3(n,k)$. We have
\begin{align}
t_3(n,k) = & {n-1\choose k-1} +\left\lfloor\frac{1}{s+1}\left( {n-1 \choose k} - \overline{d}(n,k)\right)\right\rfloor \nonumber \\
\ge & {n-1 \choose k-1}+\left\lfloor \frac{1}{s+1}\left({n-1 \choose k} -s{n-k-1 \choose k}\right) \right\rfloor \label{eq:t3} \\
\ge & {n-1 \choose k-1}+\frac{1}{s+1}\left({n-1 \choose k} -s{n-k-1 \choose k}\right) - \frac{s}{s+1}. \label{eq:t3a}
\end{align}

	\section{Proof of \Cref{thm:OddHadwigerForCompKneser}}
	\label{sec:pf}
	
	Recall that $n=sk+r$ where $s \ge 2$, $k \ge 3$ and $1\leq r\leq k-1$. In view of Lemma \ref{lem:strong1shallow} and its proof, to prove the first two statements in Theorem \ref{thm:OddHadwigerForCompKneser}, it is enough to show that for each pair of positive integers $(n, k)$ with $n \ge 2k$, one of $t_1(n,k)$, $t_2(n,k)$ and $t_3(n,k)$ is at least as big as $\chi (\overline{K}(n,k)) = \left\lceil {n \choose k}/\lfloor \frac{n}{k}\rfloor \right\rceil$. We achieve this by proving the following five lemmas.

	\begin{lemma}
		\label{lem:t1}
		We have $t_1(n,k)\geq \chi(\overline{K}(n,k))$  for $1\leq r \leq k-s$. Moreover, for $1 \leq r \leq k-2$, we have
\begin{equation}
\label{eq:diff1}
h_o(\overline {K}(2k+r,k))-\chi(\overline{K}(2k+r,k)) \geq \frac{1}{12}(1.5^{k-1}-8k-2).
\end{equation}
	\end{lemma}
	
	\begin{proof} 
		Recall from \eqref{eq:t<k-s} that 
		$$
		t_1(n,k) \geq {n-1 \choose k-1} + \frac{1}{l+1}\left({n-1 \choose k} -{n-k \choose 
			k}\right) -(k-1)\frac{l}{l+1}.
		$$ 
		Since $\chi(\overline{K}(n,k))=\lceil \frac{1}{s} {n \choose k} \rceil$ and we are working with integer values, to prove $t_1(n,k)\geq \chi(\overline{K}(n,k))$ it suffices to prove that 
		\begin{equation}
			\label{eq:prop9}
			{n-1 \choose k-1} + \frac{1}{l+1}\left({n-1 \choose k} -{n-k \choose 
				k}\right) -(k-1)\frac{l}{l+1} \geq \frac{1}{s} {n \choose k}.
		\end{equation}
		In Construction 1, we have chosen 
		$$l=\left\lceil\frac{n-k}{k-1}\right\rceil=\left\lceil\frac{sk+r-k}{k-1}\right\rceil
		%=\left\lceil\frac{sk-s+r-k+s}{k-1}\right\rceil=
		=s+\left\lceil\frac{r-k+s}{k-1}\right\rceil.$$ 
		Since $r \leq k-s$ by our assumption, we have $r-k+s\leq 0$. On the other hand, $r\geq 1$ and $s\geq 2$, so  $r-k+s\geq 3-k$. Hence, $\lceil\frac{r-k+s}{k-1}\rceil=0$ and $l=s$. Since $l = s$, ${n \choose k} = \frac{n}{k}{n-1 \choose k-1}$ and $n = sk+r$, \eqref{eq:prop9} can be presented as
		$$
		\frac{1}{s+1}\left({n-1 \choose k} -{n-k \choose 
			k}\right) -(k-1)\frac{s}{s+1} \geq  \frac{r}{n-r}{n-1 \choose k-1}.
		$$
		Replacing ${n-1 \choose k}$ by $\frac{n-k}{k}{n-1 \choose k-1}$ and then multiplying both sides by $(s+1)$, \eqref{eq:prop9} can be restated as
		\begin{equation}
			\left(s-1-\frac{r}{sk}\right) {n-1 \choose k-1} \geq {n-k \choose k} +(k-1)s.
			\label{equ:T2}
		\end{equation} 
		
		To prove (\ref{equ:T2}), we consider the following three cases one by one.
		
		\paragraph{Case I. $s=2$.} In this case, we have $n=2k+r$ and (\ref{equ:T2}) can be rewritten as
		$$
		\left(1-\frac{r}{2k}\right) {2k+r-1 \choose k-1} \geq {k+r \choose k} +2(k-1)\label{s=2},
		$$
		or, equivalently, 
		$$
		(2k-r)(2k+r-1)(2k+r-2)\cdots(k+r+1) \geq 2(k+r)(k+r-1)\cdots(r+1)+4(k-1)k!.
		$$
		Since $k\geq 3$ and $r\geq 1$, we have $(k+r)(k+r-1)\cdots(r+1) \ge 4k!$. Replacing $4k!$ by $(k+r)(k+r-1)\cdots(r+1)$ on the right-hand side of the inequality above, it suffices to prove
		%We know that $2(k+r)(k+r-1)\cdots(r+1)+4(k-1)k!\leq 2(k+r)(k+r-1)\cdots(r+1)+4(k-1)(k+r)(k+r-1)\cdots(r+1)=(4k-2)(k+r)(k+r-1)\cdots(r+1)$, so it suffices to prove that 		
		$$
		(2k-r)(2k+r-1)(2k+r-2)\cdots(k+r+1) \geq (k+1)(k+r)(k+r-1)\cdots(r+1),
		$$
		or, equivalently, 
		\begin{equation}
			\frac{(2k-r)}{k+1}\frac{(2k+r-1)}{k+r}\frac{(2k+r-2)}{k+r-1}\cdots\frac{(k+r+1)}{r+2} \geq (r+1). \label{k>=11}
		\end{equation}
		Since $r\leq k-2$, we have $\frac{2k-r}{k+1}>1$ and $\frac{k+r+i}{r+1+i}\geq \frac{3}{2}$ for $1\leq i \leq k-1$. So the left-hand side of \eqref{k>=11} is greater than $(\frac{3}{2})^{k-1} \ge k-1 \ge r+1$ for $k\geq 3$. Thus, \eqref{k>=11} holds for $k \ge 3$. Hence, we have $t_1(n,k)\geq \chi(\overline{K}(n,k))$ when $s=2$.

Now we prove \eqref{eq:diff1} before moving on to other cases. Since $s=2$, we have 	$n = 2k+r$, $1 \leq r \leq k-2$, $l = 2$, and 
$$
\chi(\overline{K}(2k+r,k))=\left\lceil \frac{1}{2} {2k+r \choose k} \right\rceil \leq  \frac{1}{2} \left({2k+r \choose k}+1\right).
$$ 
Also, by Construction 1, $h_o(\overline {K}(2k+r,k)) \ge t_1(2k+r,k)$, and $t_1(2k+r,k)$ is bounded from below by \eqref{eq:t<k-s}. Therefore, 
\begin{align*}
      &\ h_o(\overline {K}(2k+r,k)) - \chi(\overline{K}(2k+r,k)) \\
\ge &\ {2k+r-1 \choose k-1}+\frac{1}{3}\left({2k+r-1 \choose k}-{k+r \choose k}\right)-\frac{2}{3}(k-1)-\frac{1}{2}{2k+r \choose k} - \frac{1}{2} \\
=    &\ {2k+r-1 \choose k-1}+\frac{1}{3}\left(\frac{k+r}{k}{2k+r-1 \choose k-1}-{k+r \choose k}\right)-\frac{2}{3}(k-1)-\frac{2k+r}{2k}{2k+r-1 \choose k-1} - \frac{1}{2} \\
=    &\ \frac{2k-r}{6k}{2k+r-1 \choose k-1}-\frac{1}{3}{k+r \choose k}-\frac{2}{3}(k-1)-\frac{1}{2} \\
=    &\ \left(\frac{2k-r}{6(k+r)}\frac{2k+r-1}{k+r-1}\frac{2k+r-2}{k+r-2}\cdots\frac{k+r+1}{r+1}-\frac{1}{3}\right){k+r \choose k}-\frac{2}{3}(k-1)-\frac{1}{2} \\
\ge &\ \left(\frac{2k-r}{6(k+r)}\frac{2k+r-1}{k+r-1}\frac{2k+r-2}{k+r-2}\cdots\frac{k+r+1}{r+1}-\frac{1}{3}\right)-\frac{2}{3}(k-1)-\frac{1}{2} \\
\ge &\ \left(\frac{1}{12} \left(\frac{3}{2}\right)^{k-1}-\frac{1}{3}\right)-\frac{2}{3}(k-1)-\frac{1}{2} \\
=    &\ \frac{1}{12}\left(\left(\frac{3}{2}\right)^{k-1}-8k-2\right),
\end{align*}
as claimed in \eqref{eq:diff1}, where in the second last step we used the fact that $\frac{2k-r}{6(k+r)} \ge \frac{1}{12}$ and $\frac{k+r+i}{r+i} \ge \frac{3}{2}$ for $1 \le i \le k-1$.
		
		\paragraph{Case II. $s=3$.}  Since $1\leq r\leq k-3$, we have $k\geq 4$. If $k=4$, then $r=1$ and in this case we can easily verify that inequality (\ref{equ:T2}) holds. In what follows we assume $k\geq 5$. Since $n=3k+r$, (\ref{equ:T2}) is simplified to
		\begin{equation}
			\left(2-\frac{r}{3k}\right) {3k+r-1 \choose k-1} \geq {2k+r \choose k} +3(k-1), \label{s=3}
		\end{equation}
		or, equivalently, 		
		$$
		(6k-r)(3k+r-1)(3k+r-2)\cdots(2k+r+1) \geq 3(2k+r)(2k+r-1)\cdots(k+r+1)+9(k-1)k!.
		$$
		Since $k\geq 5$, we have $9(k-1)k!\leq (2k+r-1)(2k+r-2)\cdots(k+r+1)$. Hence it suffices to prove that 
		$(6k-r)(3k+r-1)(3k+r-2)\cdots(2k+r+1) \geq (6k+3r+1)(2k+r-1)(2k+r-2)\cdots(k+r+1).$ This is to say that
		\begin{equation}
			\frac{(6k-r)}{6k+3r+1}\frac{(3k+r-1)}{2k+r-1}\frac{(3k+r-2)}{2k+r-2}\cdots\frac{(2k+r+1)}{k+r+1} \geq 1.\label{k>=5}
		\end{equation}
		Since $k\geq 5$ and $r\leq k-3$, for the first fraction on the left-hand side we have $\frac{6k-r}{6k+3r+1}\geq \frac{5}{9}$, and each of the remaining $k-1$ fractions is at least $\frac{4}{3}$. Since $\frac{5}{9}(\frac{4}{3})^{k-1}\geq 1$ when $k\geq 5$, it follows that \eqref{k>=5} holds for $k\geq 5$. Therefore, we have proved that $t_1(n,k)\geq \chi(\overline{K}(n,k))$ when $s=3$.
		
		\paragraph{Case III. $s\geq 4$.}
		In this case, we have $\frac{r}{sk}< \frac{1}{4}$. Thus, to prove (\ref{equ:T2}), it suffices to show 
		$$
		\left(s-1-\frac{1}{4}\right) {n-1 \choose k-1} \geq {n-k \choose k}+(k-1)s,
		$$
		or, equivalently,		
		\begin{equation}
			k(4s-5)(n-1)(n-2) \cdots (n-k+1) \geq 4(n-k)(n-k-1)\cdots (n-2k+1)+ 4(k-1)sk!.
			\label{equ:S4}
		\end{equation}
		We can split the left-hand side of (\ref{equ:S4}) into $k(4s-5)(n-1)(n-2) \cdots (n-k+2)(n-k)$ and $k(4s-5)(n-1)(n-2) \cdots(n-k+2)$. So it suffices to prove that 
		$$
		k(4s-5)(n-1)(n-2) \cdots (n-k+2)(n-k)\geq 4(n-k)(n-k-1)\cdots (n-2k+1),
		$$ 
		$$
		k(4s-5)(n-1)(n-2) \cdots(n-k+2)\geq 4(k-1)sk!.
		$$
		The latter is true because each factor $(n-i)$ for $1 \le i \le k-2$ on the left-hand side of the inequality is at least $3k$. The former can be rewritten as
		$$
		\frac{n-2}{(n-k-1)}\cdot \frac{n-3}{(n-k-2)} \cdots\frac{n-k+2}{(n-2k+3)}\cdot \frac{k(4s-5)(n-1)}{4(n-2k+2)(n-2k+1)}\geq 1.
		$$		
		Clearly, each of the first $k-3$ fractions is no less than 1. So, to prove this inequality, it suffices to show that $k(4s-5)(n-1)\geq 4(n-2k+2)(n-2k+1)$. We observe that $\frac{n-2k+2}{k}\leq s-1$. Thus, it would be enough to prove that $(4s-5)(n-1)\geq 4(s-1)(n-2k+1)$. Replacing $n$ by $sk+r$, the inequality is simplied to $7sk+9\geq 8s+8k+r$. Since $r\leq k-s$, we only need to prove $7sk+9\geq 8s+8k+(k-s)$, which is valid because $s\geq 4$. 
	\end{proof}
	
	%%%%%%%%%%%%%%%%
	
	\begin{lemma}
		\label{lem:t2}
		We have $t_2(n,k)\geq \chi(\overline{K}(n,k))$ in the following cases:
		\begin{enumerate}[\rm (a)]
			\item $k\geq 4$ and $s\geq 7;$
			\item $k=3$ and $s\geq 17;$
			\item $k\geq 7$ and $s=6$.
		\end{enumerate}
	\end{lemma}
	
	\begin{proof}
		Recall that $l=\lceil\frac{n-1}{2(k-1)}\rceil$ in Construction 2. In view of \eqref{eq:t2}, to prove $t_2(n,k)\geq \chi(\overline{K}(n,k))=\left\lceil\frac{1}{s}\binom{n}{k}\right\rceil$, it suffices to prove that
		$$
		\frac{1}{l+1}\left({n \choose k}-{\lfloor\frac{n}{2}\rfloor \choose k}\right)-\frac{l}{l+1}\left\lceil\frac{n}{2}\right\rceil \geq \frac{1}{s} \binom{n}{k},
		$$	
		%\Longleftarrow \quad\frac{1}{l+1}\left({n \choose k}-{\lfloor\frac{n}{2}\rfloor \choose k}\right) & \geq \frac{l}{l+1}\left\lceil\frac{n}{2}\right\rceil+\frac{1}{s}\binom{n}{k}\\
		or, equivalently,
		\begin{equation}
			\label{ineq:cons2main}  
			\quad (s-l-1)\binom{n}{k} \geq s\binom{\lfloor\frac{n}{2}\rfloor}{k}+sl\left\lceil\frac{n}{2}\right\rceil.      
		\end{equation}
		
		\paragraph{Case I. $k\geq 5$.}

		We first notice that
		$$l=\left\lceil\frac{n-1}{2(k-1)}\right\rceil\leq \frac{sk+r+2k-4}{2(k-1)}\leq \frac{sk+3k-5}{2(k-1)}.$$	
		So
		\begin{align*}
			s-l-1 & \geq s-\frac{sk+3k-5}{2(k-1)}-1 \\ 
			& = \frac{2(s-1)(k-1)-sk-3k+5}{2(k-1)} \\ 
			& = \frac{sk-2s-5k+7}{2(k-1)} \\
			& = \frac{(k-2)(s-5)-3}{2(k-1)} \\
			& \geq\frac{(k-2)(s-6)}{2(k-1)},
		\end{align*}
		where the last inequality uses the fact that $k\geq 5.$
		
		It can be verified that the left-hand side of \eqref{ineq:cons2main} is no less than $\frac{(k-2)(s-6)}{2(k-1)}\binom{n}{k}$  while the right-hand side of \eqref{ineq:cons2main} is no more than $s\binom{\frac{n}{2}}{k}+\frac{sl(n+1)}{2}$. Thus, to prove \eqref{ineq:cons2main}, it suffices to show that
		$$
		\frac{(k-2)(s-6)}{2(k-1)}\binom{n}{k} \geq  s\binom{\frac{n}{2}}{k}+\frac{sl(n+1)}{2},
		$$
		or, equivalently, 
		$$
		\frac{(k-2)(s-6)}{2(k-1)}\cdot \frac{n(n-1)\cdots(n-k+1)}{k!} \geq  s\cdot\frac{\left(\frac{n}{2}\right)\left(\frac{n}{2}-1\right)\cdots\left(\frac{n}{2}-k+1\right)}{k!}+\frac{sl(n+1)}{2}.
		$$
		Multiplying both sides by $2^{k}(k-1)k!$, the latter inequality can be rewritten as
		\begin{equation*}
			(k-2)(s-6)2^{k-1} \prod_{i=0}^{k-1} (n-i) \geq (k-1)s \prod_{i=0}^{k-1} (n-2i)+2^{k-1}(k-1)k!sl(n+1).
		\end{equation*}
		By splitting the factor $(n-1)$ on the left-hand side into $(n-2k+2)+(2k-3)$, it would be enough to show the following two inequalities:
		\begin{equation}
			\label{ineq:cons2I1}
			(k-2)(s-6)2^{k-1} n(n-2k+2) \prod_{i=2}^{k-1} (n-i) \geq  (k-1)s \prod_{i=0}^{k-1} (n-2i)
		\end{equation}
		\begin{equation}
			\label{ineq:cons2I2}
			(k-2)(s-6)2^{k-1} n(2k-3) \prod_{i=2}^{k-1} (n-i) \geq 2^{k-1}(k-1)k!sl(n+1). 
		\end{equation}
		
		In fact, it is easy to see that
		$$
		(n-2k+2)\prod_{i=2}^{k-1} (n-i) \geq \prod_{i=1}^{k-1} (n-2i).
		$$
		So, to prove (\ref{ineq:cons2I1}), it suffices to show $ (k-2)(s-6)2^{k-1} n \geq  (k-1)sn$, or, equivalently,
		$$
		\frac{s}{s-6} \leq \frac{k-2}{k-1} \cdot 2^{k-1},
		$$
		which easily holds for $s\geq 7$ and $k\geq 5$ (one can see that the left-hand side is monotonically decreasing with $s$ and the right-hand side is monotonically increasing with $k$). Hence inequality (\ref{ineq:cons2I1}) holds. 
		
		Now we prove inequality (\ref{ineq:cons2I2}), which is equivalent to 
		$$(k-2)(s-6)n(2k-3)(n-2)(n-3)\cdots(n-k+1)\geq (k-1)k!sl(n+1).$$
		Note that
		$$l=\left\lceil\frac{n-1}{2(k-1)}\right\rceil\leq \frac{n-1+2k-3}{2(k-1)}\leq \frac{n+2k-4}{2(k-1)}.$$
		Thus, to prove \eqref{ineq:cons2I2}, it suffices to prove that 
		$$
		2(k-2)(s-6)n(2k-3)(n-2)(n-3)(n-4) \cdots (n-k+1) \geq k!s(n+2k-4)(n+1).
		$$
		We break this task into three pieces. First, it is easy to see
		$(n-5)(n-6)\cdots(n-k+1)\geq k(k-1)\cdots 7\cdot 6$ by comparing factors. Second, since $n\geq sk\geq 7k$, we see that $(n-2)(n-3)(n-4)\geq (n+2k-4)(n+1)$. Finally, since $k\geq 5$ and $s\geq 7$, we have $2(k-2)(2k-3)(s-6)n\geq 120s$. Putting these together, we obtain \eqref{ineq:cons2I2} immediately.

		\paragraph{Case II. $k=4$.}	
		\subparagraph{II.1. $n=4s+1$.}
		We note again an upper bound for $l$:
		$$
		l=\left\lceil\frac{4s+1-1}{6}\right\rceil=\left\lceil\frac{2s}{3}\right\rceil\leq \frac{2s+2}{3}.
		$$		
		Hence, it suffices to show that
		$$
		(s-l-1)\binom{4s+1}{4} \geq s\binom{\lfloor\frac{4s+1}{2}\rfloor}{4}+sl\left\lceil\frac{4s+1}{2}\right\rceil,
		$$
		which holds if
		$$
		\left(s-\frac{2s+2}{3}-1\right)\binom{4s+1}{4} \geq s\binom{2s}{4}+s\cdot\frac{2s+2}{3}\cdot(2s+1).
		$$
		The latter inequality is further reduced to
		$$
		\frac{s}{18}  \left(52 s^4-316 s^3+99 s^2-5 s-22\right)\geq 0,
		$$
		which holds for $s\geq 7.$
		
		\subparagraph{II.2. $n=4s+2$.}
		In this case, we have $l = \lceil\frac{4s+1}{6}\rceil$ and so the desired inequality \eqref{ineq:cons2main} becomes
		$$
		\left(s-\left\lceil\frac{4s+1}{6}\right\rceil-1\right)\binom{4s+2}{4} \geq s\binom{2s+1}{4}+s\left\lceil\frac{4s+1}{6}\right\rceil(2s+1).
		$$
		This is reduced to
		$$\frac{s(2s+1)}{6} \left(-4 \left(8 s^2+1\right) \left\lceil \frac{4s+1}{6} \right\rceil +30 s^3-29 s^2-3 s+2\right)\geq 0.$$
		One can verify this inequality for $s=7$ by directly plugging in $s=7$. For $s\geq 8$, we may replace the ceiling by $(4s+6)/6$ and then verify the simplified inequality straightaway.
		
		\subparagraph{II.3. $n=4s+3$.} In this case, we have $l = \lceil\frac{2s+1}{3}\rceil \leq\frac{2s+3}{3}$ and the desired inequality \eqref{ineq:cons2main} becomes
		$$
		(s-l-1)\binom{4s+3}{4} \geq s\binom{\lfloor\frac{4s+3}{2}\rfloor}{4}+sl\left\lceil\frac{4s+3}{2}\right\rceil,
		$$
		which holds if
		$$
		\left(s-\frac{2s+3}{3}-1\right)\binom{4s+3}{4} \geq s\binom{2s+1}{4}+s\cdot\frac{2s+3}{3}\cdot(2s+2).
		$$
The latter inequality is reduced to
		$$\frac{s}{18}\left(52 s^4-276 s^3-553 s^2-321 s-72\right)\geq 0,$$
		which holds for $s\geq 7.$

		\paragraph{Case III. $k=3$.}

		\subparagraph{III.1 $n=3s+1$.} In this case, we have $l = \lceil\frac{3s}{4}\rceil\leq\frac{3s+3}{4}$ and so the desired inequality \eqref{ineq:cons2main} becomes 
		$$
		(s-l-1)\binom{3s+1}{3} \geq s\binom{\lfloor\frac{3s+1}{2}\rfloor}{3}+sl\left\lceil\frac{3s+1}{2}\right\rceil,
		$$
		which holds if
		$$
		\left(s-\frac{3s+3}{4}-1\right)\binom{3s+1}{3} \geq s\binom{\frac{3s+1}{2}}{3}+s\cdot\frac{3s+3}{4}\cdot\frac{3s+2}{2}.
		$$
		The latter inequality is reduced to
		$$\frac{s}{16}  \left(9 s^3-135 s^2-31 s+1\right)\geq 0,$$
		which holds for $s\geq 16.$
		
		\subparagraph{III.2 $n=3s+2$.} In this case, we have $l=\lceil\frac{3s+1}{4}\rceil\leq\frac{3s+4}{4}$ and so the desired inequality \eqref{ineq:cons2main} becomes
		\begin{equation}
			\label{ineq:cons2III2}
			(s-l-1)\binom{3s+2}{3} \geq s\binom{\lfloor\frac{3s+2}{2}\rfloor}{3}+sl\left\lceil\frac{3s+2}{2}\right\rceil,
		\end{equation}	
		which holds if		
		$$
		\left(s-\frac{3s+4}{4}-1\right)\binom{3s+2}{3} \geq s\binom{\frac{3s+2}{2}}{3}+s\cdot\frac{3s+4}{4}\cdot\frac{3s+3}{2}.
		$$
		The latter inequality is reduced to
		$$\frac{s}{16} \left(9 s^3-144 s^2-178 s-56\right)\geq 0,$$
		which holds for $s\geq 18.$ In addition, one can directly verify (\ref{ineq:cons2III2}) for $s=17$.
		
		\paragraph{Case IV. $s=6$.}	
		In this case, we have $6k\leq n\leq 7k-1$ and 
		$$
		l=\left\lceil\frac{6k+r-1}{2(k-1)}\right\rceil\leq\frac{7k-1-1+2k-3}{2(k-1)}\leq\frac{9k-5}{2(k-1)}.
		$$ 
		To prove \eqref{ineq:cons2main}, it suffices to show
		$$
		(6-l-1)\binom{6k}{k} \geq 6\binom{\lfloor\frac{7k-1}{2}\rfloor}{k}+6l\left\lceil\frac{7k-1}{2}\right\rceil,
		$$
		which holds if
		$$
		\frac{k-5}{2(k-1)}\binom{6k}{k} \geq 6\binom{\frac{7k-1}{2}}{k}+6\cdot\frac{9k-5}{2(k-1)}\cdot\frac{7k}{2}.
		$$
		Multiplying both sides by $2^{k+1}(k-1)k!$, the latter is equivalent to
		\begin{align*}
			(k-5)2^k(6k)(6k-1)\cdots(5k+1)&\geq12(k-1)(7k-1)(7k-3)\cdots(5k+1)\\
			& \quad +21k(9k-5)2^kk!.
		\end{align*}
		To conclude the proof, we only need to show the following two inequalities:
		\begin{equation}
			\label{ineq:cons2IV1}
			(k-5)2^{k}(5k)\prod_{i=0}^{k-2}(6k-i) \geq 12(k-1)\prod_{i=0}^{k-1}(7k-2i-1)
		\end{equation}
		\begin{equation}
			\label{ineq:cons2IV2}
			(k-5)2^{k}\prod_{i=0}^{k-2}(6k-i) \geq 21k(9k-5)2^kk!.
		\end{equation}
		In fact, for $k\geq 7$, we have
		\begin{align*}
			(k-5)2^{k}(5k)\prod_{i=0}^{k-2}(6k-i) &=(k-5)\left(\frac{12}{7}\right)^{k}\left(\frac{7}{6}\right)^{k}(5k)\prod_{i=0}^{k-2}(6k-i)\\&\geq (k-5)\left(\frac{12}{7}\right)^{k}\prod_{i=0}^{k-1}(7k-2i-1)\\
			&\geq 12(k-1)\prod_{i=0}^{k-1}(7k-2i-1).
		\end{align*}
		Hence inequality \eqref{ineq:cons2IV1} holds. For $k\geq 6$, we have
		$$
		(k-5)2^{k}\prod_{i=0}^{k-2}(6k-i)  \geq (k-5)2^k6^{k-1}k! \geq 21k(9k-5)2^kk!.
		$$
		Hence inequality \eqref{ineq:cons2IV2} holds. 
	\end{proof}

	\begin{lemma}
		\label{lem:t3}
		Assuming $n=sk+r$ and $k-s+1\leq r \leq k-1$, we have $t_3(n,k)\geq \chi(\overline{K}(n,k))$ in each of the following cases:
		\begin{enumerate}[\rm (a)]
			\item $s=2$ (hence $r=k-1$);
			\item $s=3$, $k\geq 4$;
			\item $s=4$, $k\geq 7$;
			\item $s=4$, $k=6$, $k-3 \leq r \leq k-2$;
			\item $s=5$, $k\geq 9$;
			\item $s=5$, $k=8$, $r=4$.
		\end{enumerate}
Moreover, when $s = 2$, we have
\begin{equation}
\label{eq:diff2}
h_o(\overline {K}(3k-1,k))-\chi(\overline{K}(3k-1,k)) \geq \frac{1}{6}(1.5^{k-1}-11).
\end{equation}
	\end{lemma}
	
\begin{proof} 
The claim to be proved is $$t_3(sk+r, k)\geq \chi(\overline{K}(sk+r,k))=\left\lceil\frac{1}{s}{sk+r \choose k}\right\rceil.$$
In view of \eqref{eq:t3a}, to prove $t_3(n,k)\geq \chi(\overline{K}(n,k))$ it would be enough to show that 
		$${sk+r-1 \choose k-1}+\frac{1}{s+1}\left({sk+r-1 \choose k} -s{(s-1)k+r-1 \choose k}\right)-\frac{s}{s+1}\geq \frac{1}{s}{sk+r \choose k}.$$
		Using the formulas ${sk+r-1 \choose k}=\frac{(s-1)k+r}{k}{sk+r-1 \choose k-1}$ and ${sk+r \choose k}=\frac{sk+r}{k}{sk+r-1 \choose k-1}$, this inequality can be rewritten as
		\begin{equation}
			\label{ineq:r>k-s+1}
			\frac{(s^2-s)k-r}{sk}{sk+r-1 \choose k-1} \geq  s{(s-1)k+r-1 \choose k}+s.
		\end{equation}
		Multiplying both sides by $sk!$, the inequality can be further rewritten as
		$$
		((s^2-s)k-r) \prod_{i=1}^{k-1} (sk+r-i) \ge s^2 \prod_{i=1}^{k} ((s-1)k+r-i) + s^2k!.
		$$
		Since $k! \le \prod_{i=1}^{k} ((s-1)k+r-i)$, it suffices to show that
		$$
		\frac{(s^2-s)k-r}{(s-2)k+r}\cdot\frac{(sk+r-1)}{ (s-1)k+r-1}\cdot\frac{(sk+r-2)}{(s-1)k+r-2}\cdots\frac{(s-1)k+r+1}{(s-2)k+r+1}\geq 2s^2.
		$$
		Since $k-s+1\leq r\leq k-1$, we have $\frac{(s^2-s)k-r}{(s-2)k+r}\geq \frac{s^2-s-1}{s-1}$ and, furthermore, $\frac{(s-1)k+r+i}{(s-2)k+r+i}\geq \frac{s+1}{s}$ for any $1\leq i\leq k-1$. Therefore, to prove the inequality above it would be enough to prove
		\begin{equation}
			\label{eq:s}
			\frac{s^2-s-1}{s-1}\left(\frac{s+1}{s}\right)^{k-1}\geq 2s^2.
		\end{equation}
		If $s=2$, the left-hand side of \eqref{eq:s} is $(\frac{3}{2})^{k-1}$ and $2s^2=8$, so the inequality holds for $k\geq 7$ (and $r=k-1$). For $k=3,4,5,6$ and $r=k-1$, we can verify inequality (\ref{ineq:r>k-s+1}) directly. If $s=3$, the left-hand side of \eqref{eq:s} is $\frac{5}{2}(\frac{4}{3})^{k-1}$ while $2s^2=18$, and the inequality holds for $k\geq 8$ and $k-2\leq r\leq k-1$. For $k=4, 5, 6, 7$ and $r=k-1, k-2$, once again we can verify inequality (\ref{ineq:r>k-s+1}) directly. 
		
		%\Meirun{The cases $s=3$, $k=3, r=1, 2$, are left as special cases.}
		
		Similarly, for $s=4$, \eqref{eq:s} is reduced to $\frac{11}{3}(\frac{5}{4})^{k-1}>2s^2=32$ which is valid for any $k\geq 11$ and $k-3\leq r\leq k-1$. Once again for $k=7, 8, 9, 10$, $r=k-1, k-2, k-3$, and $k=6, r=k-2, k-3$, we can verify inequality~(\ref{ineq:r>k-s+1}) directly. 
		
		%\Meirun{The cases $s=4$,  $k=5, r=2,3,4$ and $k=6, r=5$ are left as special cases.}
		
		%$k=3, r=1, 2$; $k=4, r=1, 2, 3$;
		
		Finally, for $s=5$, \eqref{eq:s} is reduced to $\frac{19}{4}(\frac{6}{5})^{k-1}>2s^2=50$ which holds for any $k\geq 14$ and $k-4\leq r\leq k-1$. For $k=9, 10, 11, 12, 13$, $r=k-1, k-2, k-3, k-4$, and $(k, r) = (8, k-4)$, we can verify inequality~(\ref{ineq:r>k-s+1}) directly.
		%	\Meirun{We have some special cases left: $k=5,6,7, k-4\leq r\leq k-1$; $k=8, r=5,6,7$.}
		%$k=3, r=1, 2$; $k=4,r=1,2,3$;	
		
It remains to prove \eqref{eq:diff2} in the case $s= 2$. Since $s=2$, we have $r = k-1$, $n=3k-1$, and 
$$
\chi(\overline{K}(3k-1,k))=\left\lceil \frac{1}{2} {3k-1 \choose k} \right\rceil \leq  \frac{1}{2} \left({3k-1 \choose k}+1\right).
$$ 
Also, by Construction 3, $h_o(\overline {K}(3k-1,k)) \ge t_3(3k-1,k)$, and $t_3(3k-1,k)$ is bounded from below by \eqref{eq:t3a}. Thus, 
\begin{align*}
      &\ h_o(\overline {K}(3k-1,k)) - \chi(\overline{K}(3k-1,k)) \\
\ge &\ {3k-2 \choose k-1}+\frac{1}{3}\left({3k-2 \choose k}-2{2k-2 \choose k}\right)-\frac{2}{3}-\frac{1}{2}{3k-1 \choose k}-\frac{1}{2} \\
= &\ {3k-2 \choose k-1}+\frac{1}{3}\left(\frac{2k-1}{k}{3k-2 \choose k-1}-2{2k-2 \choose k}\right)-\frac{2}{3}-\frac{1}{2}\frac{3k-1}{k}{3k-2 \choose k-1}-\frac{1}{2} \\
= &\ \frac{1}{6} \frac{k+1}{k}{3k-2 \choose k-1}-\frac{2}{3}{2k-2 \choose k}-\frac{7}{6} \\
= &\ \left(\frac{1}{6}\frac{k+1}{k-1} \frac{3k-2}{2k-2}\frac{3k-3}{2k-3}\cdots\frac{2k}{k}-\frac{2}{3}\right){2k-2 \choose k}-\frac{7}{6} \\
\ge &\ \frac{1}{6}\frac{k+1}{k-1} \frac{3k-2}{2k-2}\frac{3k-3}{2k-3}\cdots\frac{2k}{k}-\frac{2}{3}-\frac{7}{6} \\
\ge &\ \frac{1}{6}\left(\left(\frac{3}{2}\right)^{k-1}-11\right),
\end{align*}
as desired by \eqref{eq:diff2}, where in the last step we used the fact that $\frac{2k+i-1}{k+i-1} \ge \frac{3}{2}$ for $1 \le i \le k-1$.
	\end{proof}
	
	We deal with the case $(n,k)=(17,4)$ separately in the following lemma as the precise value of $\overline{d}(17,4)$ is needed in this case.
	
	\begin{lemma}
		\label{lem:17-4}
		We have $t_3(17,4)\geq \chi(\overline{K}(17,4))$.
	\end{lemma}	
	
	\begin{proof} 
		By the inclusion-exclusion principle, we know that $\overline{d}(17,4)=4{12 \choose 4}-6{8 \choose 4}+4{4 \choose 4}$. 
		Thus, by \eqref{eq:t3}, we have $t_3(17,4)\geq{16 \choose 3}+\lfloor\frac{1}{5}({16 \choose 4}- 4{12 \choose 4}+6{8 \choose 4}-4{4 \choose 4})\rfloor=560+51=611\geq \chi(\overline{K}(17,4))=595$.		
	\end{proof}
	
	\begin{lemma}
		\label{lem:NotCoveredCases}
		We have $t_1(n,k)\geq \chi(\overline{K}(n,k))$ for all pairs $(n, k)$ not covered by Lemmas \ref{lem:t1}--\ref{lem:17-4}. 
	\end{lemma}
	
	\begin{proof}
		For all pairs $(n, k)$ not covered by Lemmas \ref{lem:t1}--\ref{lem:17-4}, we can directly compute the values of $t_1(n,k)$ and $\chi (\overline{K}(n,k))$ using \eqref{eq:t<k-s} and Lemma \ref{lem:Baranyai}, respectively. These values are given in Table~\ref{tab:Summary} from which we obtain $t_1(n,k)\geq \chi(\overline{K}(n,k))$ in all cases.
	\end{proof}
	
		{\small	
		\renewcommand{\arraystretch}{1.40} 
	
\begin{table}[ht]
	
	\begin{tabular}{|l||*{10}{c|}}\hline
		$(n,k)$
		%\backslashbox{$r$}{$s$}
		&\makebox[2em]{(10,3)}&\makebox[2em]{(11,3)}
		&\makebox[2em]{(13,3)}&\makebox[2em]{(14,3)}&\makebox[2em]{(16,3)}&\makebox[2em]{(17,3)}&\makebox[2em]{(19,3)}&\makebox[2em]{(20,3)}&\makebox[2em]{(22,3)}&\makebox[2em]{(23,3)}\\\hline
		$t_1(n,k)$ &45&57&82&94&125&144&181&199&242&267\\\hline
		$ \chi(\overline{K}(n,k))$ &40&55&72&91&112&136&162&190&220&253\\\hline
	\end{tabular}
	
	\smallskip	
	\begin{tabular}{|l||*{10}{c|}}\hline
		$(n,k)$
		%\backslashbox{$r$}{$s$}
		&\makebox[2em]{(25,3)}&\makebox[2em]{(26,3)}
		&\makebox[2em]{(28,3)}&\makebox[2em]{(29,3)}&\makebox[2em]{(31,3)}&\makebox[2em]{(32,3)}&\makebox[2em]{(34,3)}&\makebox[2em]{(35,3)}&\makebox[2em]{(37,3)}&\makebox[2em]{(38,3)}\\\hline
		$t_1(n,k)$ &316&340&395&426&487&517&584&621&694&730\\\hline
		$\chi(\overline{K}(n,k))$ &288&325&364&406&450&496&544&595&648&703\\\hline
	\end{tabular}
	
	\smallskip	
	\begin{tabular}{|l||*{10}{c|}}\hline
		$(n,k)$
		%\backslashbox{$r$}{$s$}
		&\makebox[2em]{(40,3)}&\makebox[2em]{(41,3)}
		&\makebox[2em]{(43,3)}&\makebox[2em]{(44,3)}&\makebox[2em]{(46,3)}&\makebox[2em]{(47,3)}&\makebox[2em]{(49,3)}&\makebox[2em]{(50,3)}&\makebox[2em]{(18,4)}&\makebox[2em]{(19,4)}\\\hline
		$t_1(n,k)$ &809&852&937&979&1070&1119&1216&1264&908&1097\\\hline
		$\chi(\overline{K}(n,k))$ &760&820&882&946&1012&1081&1152&1225&765&969\\\hline
	\end{tabular}
	
		\smallskip	
	\begin{tabular}{|l||*{10}{c|}}\hline
		$(n,k)$
		%\backslashbox{$r$}{$s$}
		
		&\makebox[2em]{(21,4)}&\makebox[2em]{(22,4)}&\makebox[2em]{(23,4)}&\makebox[2em]{(25,4)}&\makebox[2em]{(26,4)}&\makebox[2em]{(27,4)}&\makebox[2em]{(22,5)}&\makebox[2em]{(23,5)}&\makebox[2em]{(24,5)}&\makebox[2em]{(26,5)}\\\hline
		$t_1(n,k)$ &1491&1746&1969&2603&2891&3275&8344&10275&12524&17333\\\hline
		$\chi(\overline{K}(n,k))$ &1197&1463&1771&2109&2492&2925&6584&8413&10626&13156\\\hline
	\end{tabular}

	\smallskip	
	\begin{tabular}{|l||*{10}{c|}}\hline
		$(n,k)$
		%\backslashbox{$r$}{$s$}
		
		&\makebox[2em]{(27,5)}&\makebox[2em]{(28,5)}&\makebox[2em]{(29,5)}&\makebox[2em]{(31,5)}&\makebox[2em]{(32,5)}&\makebox[2em]{(33,5)}&\makebox[2em]{(34,5)}&\makebox[2em]{(29,6)}&\makebox[2em]{(32,6)}\\\hline
		$t_1(n,k)$ &20585&24275&28442&36993&42610&48845&54095&137161&242203\\\hline
		$\chi(\overline{K}(n,k))$ &16146&19656&23751&28319&33562&39556&46376&118755&181239\\\hline
	\end{tabular}
		\smallskip
	\begin{tabular}{|l||*{11}{c|}}\hline
		$(n,k)$
		%\backslashbox{$r$}{$s$}
		&\makebox[2em]{(33,6)}
		&\makebox[2em]{(34,6)}&\makebox[2em]{(35,6)}&\makebox[2em]{(37,6)}&\makebox[2em]{(38,6)}&\makebox[2em]{(39,6)}&\makebox[2em]{(40,6)}&\makebox[2em]{(41,6)}\\\hline
		$t_1(n,k)$ &288544&341740&402525&528430&613221&708581&815472&934909
		\\
		\hline
		$\chi(\overline{K}(n,k))$ &221514&268981&324632&387464&460114&543771&639730&749398
		\\
		\hline
	\end{tabular}
	
		\smallskip
	\begin{tabular}{|l||*{8}{c|}}\hline
		$(n,k)$
		%\backslashbox{$r$}{$s$}
		&\makebox[2em]{(38,7)}&\makebox[2em]{(39,7)}&\makebox[2em]{(40,7)}
		&\makebox[2em]{(41,7)}&\makebox[2em]{(45,8)}&\makebox[2em]{(46,8)}&\makebox[2em]{(47,8)}\\\hline
		$t_1(n,k)$ &3419912&4082738&4849607&5733229&58124078&69186715&82011684\\\hline
		$\chi(\overline{K}(n,k))$ &2524052&3076188&3728712&4496388&43110639&52186563&62891499\\\hline
	\end{tabular}
	\caption{Cases not covered by Lemmas  \ref{lem:t1}--\ref{lem:17-4}.}
	\label{tab:Summary}
\end{table}
	}

	\medskip
	\Proof~\ref{thm:OddHadwigerForCompKneser}. 
The first two statements in Theorem \ref{thm:OddHadwigerForCompKneser} follow from Lemmas  \ref{lem:divisor}--\ref{lem:2}, Constructions 1--3, and Lemmas \ref{lem:t1}--\ref{lem:NotCoveredCases} with no further effort. The third statement follows from \eqref{eq:diff1} and \eqref{eq:diff2} immediately. 
	\qed

\bigskip
\noindent \textbf{Acknowledgments}~~~This work has received support under the program ``Investissement d'Avenir" launched by the French Government and implemented by ANR, with the reference ``ANR‐18‐IdEx‐0001" as part of its program ``Emergence". Meirun Chen is supported by Fujian Provincial Department of Science and Technology (2024J011197). Lujia Wang is supported by NSFC (No. 12371359). Sanming Zhou was supported by ARC
Discovery Project DP250104965.

\end{document}